\documentclass[12pt]{article}

\usepackage{amsmath}
\usepackage{amssymb, latexsym, a4, array}
\usepackage{graphicx}
\usepackage[all]{xy}
\usepackage{multicol}
\usepackage[usenames]{color}
 
\usepackage{mathrsfs}
\usepackage{hyperref}
\usepackage{cleveref}
\newtheorem{theorem}{Theorem}[section]
\newtheorem{lemma}[theorem]{Lemma}
\newtheorem{remark}[theorem]{Remark}
\newtheorem{corollary}[theorem]{Corollary}
\newtheorem{proposition}[theorem]{Proposition}

\newtheorem{definition}[theorem]{Definition}
\newtheorem{example}[theorem]{Example}
\newenvironment{proof}{\trivlist\item[]\rm{\textbf{Proof.}\ }}{\endtrivlist}
 \makeatletter
      \def\@setcopyright{}
      \def\serieslogo@{}
      \makeatother

 \newcommand{\Ima}{\mathrm{Im}}

\newcommand{\Lie}{\ensuremath{\mathsf{Lie}}}
\newcommand{\Leib}{\ensuremath{\mathsf{Leib}}}

\newcommand{\q}{\mathfrak{q}}
\newcommand{\m}{\mathfrak{m}}
\newcommand{\s}{\mathfrak{s}}
\newcommand{\f}{\mathfrak{f}}
\newcommand{\h}{\mathfrak{h}}

\newcommand{\I}{\mathfrak{I}}
\newcommand{\R}{\mathfrak{r}}

\author{N.G. Bell Bogmis, G.R. Biyogmam, H. Safa  and  C. Tcheka \\}

\title{Upper bounds on the dimension of the Schur \Lie-multiplier of \Lie-nilpotent Leibniz $n$-algebras}

\begin{document}
\maketitle


\noindent\textbf{Abstract.}
The Schur  \Lie-multiplier of Leibniz algebras is the Schur multiplier of Leibniz algebras defined relative to the Liezation functor.
In this paper, we study upper bounds for the dimension of the Schur \Lie-multiplier of \Lie-filiform Leibniz $n$-algebras and the Schur \Lie-multiplier of its \Lie-central factor. The upper bound obtained is associated to both the sequences of central binomial coefficients and the sum of the numbers located in the rhombus part of Pascal's triangle.
Also, the pattern of counting the number of \Lie-brackets of a particular  Leibniz $n$-algebra leads us
 to a new property of Pascal's triangle.
   Moreover, we discuss some results which improve the existing upper bound published in \cite{s-b1} for $m$-dimensional \Lie-nilpotent  Leibniz $n$-algebras with $d$-dimensional \Lie-commutator. In particular, it is shown that
 if $\q$ is  an $m$-dimensional \Lie-nilpotent  Leibniz $2$-algebra with one-dimensional \Lie-commutator, then
$\dim\mathcal{M}_\Lie(\q)\leq \frac{1}{2}m(m-1)-1.$  \\




\section{Introduction}

The notion of Schur multiplier was introduced by Issai Schur  \cite{s} in 1904 in his study of projective representations of a group. The upper bound of the Schur multiplier plays an important role in the classification of  $p$-groups \cite{grn}
and the characterization of nilpotent Lie algebras  \cite{m}.  So, finding the best possible upper bound of the dimension of the Schur multiplier has been a focus of research in the last few years. Leibniz algebras first appeared in works published in  1965  by Blokh  \cite {Bl}, and  were independently rediscovered by  Jean-Louis Loday \cite{l1, l3} in his study of periodicity phenomena in algebraic $K$-theory. Leibniz algebras are essentially generalization of Lie algebras. Recently, some papers \cite{b-c2, b-c3, safa1,s-b1} have considered studying notions of Leibniz algebras relative to the {\it Liezation functor} $(-)_{\Lie}: {\sf Leib} \to {\sf Lie}$ which assigns the Lie algebra $\q_{\Lie} =\q/\Leib(\q)$ to a given Leibniz algebra $\q$, where $\Leib(\q)=\langle [x,x]|\ x \in \q\rangle$.
Translating properties concerning central extensions and commutators from the absolute case to the relative one has yielded several notions such as the Schur \Lie-multiplier, which was first studied on Leibniz algebras in \cite{c-i} and was used to characterize several notions such as \Lie-nilpotency and
\Lie-stem covers. Leibniz $n$-algebras  were introduced by Casas, Loday and Pirashvili \cite{c-l-p} as  a generalization of Leibniz algebras  to $n$-ary algebras. They also generalize the concept of Filippov algebras  \cite{Fil}  introduced  by Filippov in 1987.  In this paper, we investigate some results obtained in Leibniz algebras and Filippov algebras to Leibniz $n$-algebras. Our main purpose is to  study the upper bound of the Schur \Lie-multiplier of  \Lie-nilpotent Leibniz $n$-algebras which can be useful in classifying  both nilpotent Leibniz $n$-algebras \cite{l3} and nilpotent Filippov algebras. Note that the Schur \Lie-multiplier can be used to provide examples in which properties on Leibniz algebras cannot be extended to Leibniz $n$-algebras.  For instance, the  Schur \Lie-multiplier of a Leibniz algebra is a Lie algebra.
  In fact, a Leibniz algebra is \Lie-abelian if and only if it is a Lie algebra.
  But, this is not true for Leibniz $n$-algebras with $n\geq 3,$ as discussed in \cite[Remark 2.1]{s-b1}.

The remaining of the paper is organized as follows: In Section \ref{Preliminaries}, we present some  preliminaries and generalities, which include the definition of  the notion of Schur \Lie-multiplier of a  Leibniz $n$-algebra, and some immediate results on its dimension.  In Section \ref{Main}, we provide the main results of the paper.  In particular, we prove that the dimension of the Schur \Lie-multiplier of a \Lie-filiform Leibniz $n$-algebra differs from  the dimension of the  Schur \Lie-multiplier of its \Lie-central factor by no more than the number of  \Lie-brackets of
   a particular   Leibniz $n$-algebra.  The upper bound obtained is associated to  the sequences of central binomial coefficients and the sum of the numbers located in the rhombus part of Pascal's triangle. Moreover, the way the \Lie-brackets are counted  yields
  the following relationship between central binomial coefficients and triangular numbers:
$${2n\choose n}=\sum_{i=0}^n {2n-r-1-i\choose n-r-1}{i+r\choose r},$$
 for every $1\leq r\leq n-1$, which is actually a new property of Pascal's triangle (Remark \ref{rem1}).
 Next, we discuss some inequalities on the dimension of the Schur \Lie-multiplier of \Lie-nilpotent Leibniz $n$-algebras.
 As a result, it is shown that  for an $m$-dimensional \Lie-nilpotent  Leibniz $n$-algebra $\q$ whose \Lie-commutator is of dimension
$d$, its \Lie-multiplier dimension is bounded as
\begin{equation}\label{first}
    \dim\mathcal{M}_\Lie(\q)\leq {m-d+n-1\choose n}+ d {m-d+n-2\choose n-1}-d,
\end{equation}
which improves $\dim\mathcal{M}_\Lie(\q)\leq\sum_{i=1}^{n} {{n-1}\choose{i-1}}{m\choose i}-d$, given in \cite{s-b1}.
Furthermore, using another technique we prove that if $\q$ is  an $m$-dimensional \Lie-nilpotent
  Leibniz $2$-algebra with one-dimensional \Lie-commutator, then
$$\dim\mathcal{M}_\Lie(\q)\leq \frac{1}{2}m(m-1)-1,$$
which  improves  the upper bound (\ref{first}) (when $n=2$ and $d=1$),
and  also the upper bound   $\dim\mathcal{M}_\Lie(\q)\leq \frac{1}{2}m(m+1)-d$,  obtained in \cite{s-b1}  (when $d=1$).


\section{Preliminaries}\label{Preliminaries}

Throughout the paper, $n\geq 2$ is a fixed  integer and all Leibniz $n$-algebras are
considered over a fixed field $\mathbb{K}$ of characteristic zero.
 Recall that a Leibniz algebra \cite{l1,l3} is a vector space $\q$  equipped with a bilinear map $[-,-]: \q \times \q \to \q$, usually called the Leibniz bracket of ${\q}$,  satisfying the Leibniz identity
$$[x,[y,z]]= [[x,y],z]-[[x,z],y],$$
 for every $x, y, z \in \q$.
A Leibniz $n$-algebra \cite{c-l-p} is a vector space $\q$ together with an
$n$-linear map $[-,\ldots,-]:\q\times\cdots\times\q\to  \q$,
satisfying the following identity:
\[[[x_1,\ldots,x_n],y_2,\ldots,y_n]=\sum_{i=1}^{n} [x_1,\ldots,x_{i-1},[x_i,y_2,\ldots,y_n],x_{i+1},\ldots,x_n],\]
for all $x_i,y_j\in \q$, $1\leq i\leq n$ and $2\leq j\leq n$.
A subspace $\h$ of a Leibniz $n$-algebra $\q$ is called a subalgebra, if $[x_1,\ldots,x_n]\in\h$ for any $x_i\in\h$.
Also, a subalgebra $\I$ of $\q$ is said to be an ideal, if the brackets $[\I,\q,\ldots,\q]$, $[\q,\I,\q,\ldots,\q]$,
$\ldots$ ,  $[\q,\ldots,\q,\I]$ are all contained in $\I$.

Let $\q$ be a Leibniz $n$-algebra. We define the bracket
$[-,\ldots,-]_{\Lie}:\q\times\cdots\times\q\to \q$ as
 $$[x_1,\ldots,x_n]_{\Lie}=\sum_{\sigma\in S_n}[x_{\sigma(1)},\ldots,x_{\sigma(n)}],$$
  where $S_n$ is the symmetric group
   of degree $n$. In other words
\begin{align*}
 [x_1,\ldots,x_n]_{\Lie}&=[(x_1 +\cdots+x_n),\ldots,(x_1 +\cdots+x_n)]\\
 &-\sum_{1\leq i_j\leq n}[x_{i_1},\ldots,x_{i_j},\ldots,x_{i_j},\ldots,x_{i_n}].
\end{align*}
 Note that the above sigma is the sum of all
possible brackets in which {\it at least} two arguments coincide. More precisely, this sum
is  $\sum_{\theta}[x_{\theta(1)},\ldots,x_{\theta(n)}]$
  where $\theta$ rides over the set of all non-injective maps $\theta:\{1,\ldots,n\}\to \{1,\ldots,n\}$ (see \cite{s-b1,s-b} for more details).
In this approach, we actually consider the relative notion of  commutator when the semi-abelian category \cite{c-v}
is the category ${\sf _nLeib}$ of Leibniz
$n$-algebras and its Birkhoff subcategory \cite{c-i} is ${\sf _nLie},$ the category of $n$-Lie algebras, together with the "Liezation-like" functor
 $(-)_{\Lie}:  {\sf _nLeib} \to  {\sf _nLie}$ which assigns the $n$-Lie algebra $\q/{\sf _nLeib(\q)}$ to a given Leibniz  $n$-algebra $\q,$ where
 $${\sf _nLeib(\q)}=\left\langle [x_1,\ldots,x_i,\ldots,x_j,\ldots, x_n]~|~\exists i,j : x_i=x_j~\mbox{with}~ x_1,\ldots, x_n\in\q \right\rangle.$$

The $\Lie$-center and  $\Lie$-commutator of a Leibniz $n$-algebra $\q$ are defined as
$Z_{\Lie}(\q)=\{x\in\q| \ [x,\q,\ldots,\q]_\Lie=0\}$  and
$\q_{\Lie}^2=[\q,\ldots,\q]_\Lie=\langle [x_1,\ldots,x_n]_{\Lie} |\ x_i\in \q\rangle$, repectively. One may easily check that these  are  ideals of $\q$
(see \cite[Remark 2.1]{s-b}).
 A Leibniz $n$-algebra $\q$ is \Lie-abelian if $\q_{\Lie}^2=0$ or equivalently $Z_{\Lie}(\q)=\q$.
Also, $\q$ is \Lie-nilpotent of class $c$ if $\q_{\Lie}^{c+1}=0$ and $\q_{\Lie}^{c}\not=0$, where $\q_{\Lie}^1=\q$,
 and  $\q_{\Lie}^{i}=[\q_{\Lie}^{i-1},\q,\ldots,\q]_\Lie$ for $i\geq 2$. 

 \begin{remark}\normalfont
 In \cite{s-b1} the $\Lie$-commutator $[\q,\ldots,\q]_\Lie$  is denoted by $\q_{\Lie}^n$, since
  we did not deal with the concept of \Lie-nilpotency there.
 \end{remark}

Let $0\rightarrow\R\rightarrow\f\rightarrow\q\rightarrow 0$  be a free presentation of a Leibniz $n$-algebra $\q$
(see \cite{cas}). Following \cite{s-b1} the  Schur \Lie-multiplier
of $\q$ is defined as
\[\mathcal{M}_\Lie(\q)=\dfrac{\R\cap\f_\Lie^2}{[\R,\f,\ldots,\f]_\Lie}.\]
Clearly, the Schur \Lie-multiplier is  independent of
the chosen free presentation of  $\q$.
 If $n=2$, then $\mathcal{M}_\Lie(\q)=(\R\cap[\f,\f]_\Lie)/[\R,\f]_\Lie$ is  the Schur \Lie-multiplier of a Leibniz algebra
  given in \cite{c-i}.
Three other generalizations of the Schur multiplier, namely the $c$-nilpotent Schur \Lie-multiplier of a Leibniz algebra,
 the Schur \Lie-multiplier of a pair of Leibniz algebras, and the Schur multiplier of a pair of Leibniz algebras
  have already been studied in \cite{b-c3,safa1,b-s}, respectively. \\


The following results are needed in the next section.

\begin{lemma}\cite[Corollary 4.3]{s-b1}\label{lem1}
Let  $\q\cong \f/\R$ be  a finite dimensional Leibniz $n$-algebra,
 $\I$ be  an ideal of $\q$ and $\s$ be an ideal of $\f$ such that $\I\cong\s/\R$.  Then
\begin{itemize}
\item[$(i)$] $\dim\mathcal{M}_\Lie(\dfrac{\q}{\I}) \leq \dim\mathcal{M}_\Lie(\q) + \dim\dfrac{\I \cap\q_\Lie^2}{[\I,\q,\ldots,\q]_\Lie}$,
\item[$(ii)$] $\dim\mathcal{M}_\Lie(\q) + \dim(\I \cap\q_\Lie^2) = \dim\mathcal{M}_\Lie(\dfrac{\q}{\I}) +
\dim\dfrac{[\s,\f,\ldots,\f]_\Lie}{[\R,\f,\ldots,\f]_\Lie}$,
\item[$(iii)$] $\dim\mathcal{M}_\Lie(\q) + \dim\q_\Lie^2 = \dim\dfrac{\f_\Lie^2}{[\R,\f,\ldots,\f]_\Lie}$,
\item[$(iv)$] if $\I\subseteq Z_{\Lie}(\q)$, then
\begin{eqnarray*}
\dim\mathcal{M}_\Lie(\q) + \dim(\I \cap\q_\Lie^2) &\leq& \dim\mathcal{M}_\Lie(\dfrac{\q}{\I}) +
\dim\big{(}\I\otimes^{n-1}\dfrac{\q}{_n{\Leib(\q)}}\big{)}\\
&\leq&\dim\mathcal{M}_\Lie(\dfrac{\q}{\I}) +
\dim\big{(}\I\otimes^{n-1}\dfrac{\q}{\q_\Lie^2}\big{)},
\end{eqnarray*}
where  $\I\otimes^{n-1}\bar{\q}=\I\otimes \bar{\q}\otimes\cdots\otimes \bar{\q}$ including $(n-1)$-times $\bar{\q}$.
\end{itemize}
\end{lemma}


\begin{theorem}\cite[Corollary 4.12]{s-b1}\label{th0}
Let $\q$ be an $m$-dimensional Leibniz $n$-algebra. Then
\[\dim\mathcal{M}_\Lie(\q)\leq\sum_{i=1}^{n} {{n-1}\choose{i-1}}{m\choose i}-\dim\q_\Lie^2.\]
\end{theorem}

\begin{theorem}\cite[Corollary 4.10]{s-b1}\label{theorem2.4}
Let $\q$ be an $m$-dimensional Leibniz $2$-algebra. Then
\[\dim\mathcal{M}_\Lie(\q)\leq\frac{1}{2}m(m+1)-\dim\q_\Lie^2.\]
Also, $\dim\mathcal{M}_\Lie(\q)=\frac{1}{2}m(m+1)$ if and only if $\q$ is \Lie-abelian.
\end{theorem}


\section{Main results}\label{Main}

This section is devoted to obtain some inequalities on the dimension of the Schur \Lie-multiplier of \Lie-nilpotent Leibniz $n$-algebras.

An $m$-dimensional Leibniz algebra $\q$ is said to be {\it filiform} if
$\dim\q^i=m-i$ for every $2\leq i\leq m$,
and is called {\it nulfiliform}  if $\dim\q^i=m+1-i$ for every $1\leq i\leq m+1$,
 where $\q^i$ is the $i$th term of the lower central series of $\q$ (see \cite{alb} for more details).
The first definition agrees with the notion of filiform Lie algebras \cite{ver}.
Following \cite{safa1}, an $m$-dimensional (non-Lie) Leibniz algebra  $\q$ is  called \Lie-filiform if
$\dim\q_{\Lie}^i=m-i$ for every $2\leq i\leq m$, and is called {\it \Lie-nilpotent of maximal class} if
$\dim\q_{\Lie}^i=m+1-i$ for every $1\leq i\leq m+1$.

\begin{definition}\label{def1} \normalfont
 An $m$-dimensional  Leibniz $n$-algebra $\q$ ($m\geq n$) is called
 filiform (resp. \Lie-filiform) if $\dim \q^i=m-n+2-i$ (resp. $\dim \q_{\Lie}^i=m-n+2-i$) for every $2\leq i\leq m-n+2$.
\end{definition}
Note that the notion of filiform Leibniz $n$-algebras agrees with the definition of filiform $n$-Lie algebras discussed in \cite{goz,safa}.
It is easy to see that a filiform (resp. \Lie-filiform)  Leibniz $n$-algebra
 is nilpotent (resp. \Lie-nilpotent) of class  $m-n+1$. In  $n$-Lie algebras, this nilpotency class is maximal (\cite[Example 3.12]{safa}),
 but if $\q$ is a Leibniz $n$-algebra (which is not an $n$-Lie algebra), then $m-n+1$  is not  the maximal (\Lie-)nilpotency class of
$\q$ (see also \cite[Remark 3.2]{safa1}). In fact, the concept of "maximal class" for Leibniz $n$-algebras should be defined similar to the case where $n=2$.

\begin{definition}\label{def2} \normalfont
An $m$-dimensional  Leibniz $n$-algebra $\q$  is said to be nilpotent (resp. \Lie-nilpotent) of maximal class, if
 $\dim \q^i=m+1-i$ (resp. $\dim \q_{\Lie}^i=m+1-i$) for every $1\leq i\leq m+1$.
\end{definition}

\begin{example}\label{ex1} \normalfont
It is easy to see that $\q$ with a basis $\{x_1,\ldots,x_m\}$ ($m\geq 2$)
and non-zero multiplications $[x_i,x_1,x_1]=x_{i+1}$ for $1\leq i\leq m-1$, is a Leibniz 3-algebra.
Now since
\[[[[x_1,\underbrace{x_1,x_1],x_1,x_1],\ldots,x_1,x_1]}_{m-\text{times}\ x_1,x_1}=[[x_2,x_1,x_1],\ldots,x_1,x_1]=\cdots=[x_m,x_1,x_1]=0\]
and
\[[[[x_1,\underbrace{x_1,x_1],x_1,x_1],\ldots,x_1,x_1]}_{(m-1)-\text{times}\ x_1,x_1}=[[x_2,x_1,x_1],\ldots,x_1,x_1]=\cdots=x_m\not=0,\]
we get $\q^{m+1}=0$ and hence $\q$ is nilpotent of class $m$. Also, $\{x_i,\ldots,x_m\}$ is a basis for $\q^i$ and hence
 $\dim\q^i=m+1-i$, for every $1\leq i\leq m+1$. Therefore, $\q$ is nilpotent  of maximal class.
 One may easily check that $\q$ is also \Lie-nilpotent  of maximal class.
 Note that
 $$[x_i,x_1,x_1]_\Lie=2[x_i,x_1,x_1]+2[x_1,x_i,x_1]+2[x_1,x_1,x_i]=2x_{i+1},$$
  and hence
 $\{x_i,\ldots,x_m\}$ is a basis for  $\q_{\Lie}^i$, as well.
\end{example}


  The following result is a generalization of \cite[Theorem 3.4]{safa1} (when $\m=\q$), and \cite[Corollary 4.7]{b-c3} (when $c=1$). See also \cite[Remark 3.2]{safa1}.

\begin{theorem}\label{th1}
$(i)$ If $\q$ is  an $m$-dimensional \Lie-filiform Leibniz $n$-algebra $(m>n)$,
then
\[\dim\mathcal{M}_\Lie(\q) =\dim\mathcal{M}_\Lie(\frac{\q}{Z_\Lie(\q)})+k,\]
such that $-1\leq k\leq {2n-2\choose n-1}-1$.

$(ii)$ If $\q$ is a finite dimensional \Lie-nilpotent Leibniz $n$-algebra  of maximal class,
then the above equality holds for $-1\leq k\leq 0$.
\end{theorem}
\begin{proof}
(i) Since $\q$ is \Lie-filiform, then it is \Lie-nilpotent of class $m-n+1$,  $\dim\q_{\Lie}^{m-n+1}=1$
and $0\not=\q_{\Lie}^{m-n+1}\subseteq Z_{\Lie}(\q)$. Consider $0\not=x\in \q_{\Lie}^{m-n+1}$, $x=[y,z_2,z_3,\ldots,z_n]$ with
$y\in \q_{\Lie}^{m-n}$ and
$z_i\in\q$, $2\leq i\leq n$. Moreover $\dim \q_{\Lie}^{m-n}=2$, thus one has $\q_{\Lie}^{m-n+1}\subseteq Z_{\Lie}(\q)\subsetneq \q_{\Lie}^{m-n}$
since $y\not\in Z_{\Lie}(\q)$. Hence $\q_{\Lie}^{m-n+1}= Z_{\Lie}(\q)$ and $\dim(\q_\Lie^2\cap Z_{\Lie}(\q))=\dim Z_{\Lie}(\q)=1$.
Now by considering Lemma \ref{lem1} (i)  in which $\I=Z_{\Lie}(\q)$, we have
$$\dim\mathcal{M}_\Lie(\dfrac{\q}{Z_{\Lie}(\q)})-1 \leq \dim\mathcal{M}_\Lie(\q).$$
On the other hand, using Lemma \ref{lem1} (ii) we get
$$\dim\mathcal{M}_\Lie(\q)  = \dim\mathcal{M}_\Lie(\dfrac{\q}{Z_{\Lie}(\q)}) +
\dim\dfrac{[\s,\f,\ldots,\f]_\Lie}{[\R,\f,\ldots,\f]_\Lie}-1.$$
It is easy to see that
$$\psi:Z_{\Lie}(\q) \times \frac{\q}{\q_{\Lie}^{2}}\times\cdots\times \frac{\q}{\q_{\Lie}^{2}}\longrightarrow \frac{[\s,\f,\ldots,\f]_\Lie}{[\R,\f,\ldots,\f]_\Lie}$$
given by $\psi(z,\bar{q}_1,\ldots,\bar{q}_{n-1})= [s,f_1,\ldots,f_{n-1}]_\Lie+[\R,\f,\ldots,\f]_\Lie$ where
$\pi(s)=z$ and $\pi(f_i)=q_i$, is a surjective
multilinear map ($\pi:\f\to\q$ is the canonical epimorphism). Since
$\dim Z_{\Lie}(\q)=1$, $\dim (\q/\q_{\Lie}^{2})=n$, and by definition of the \Lie-bracket
 $$[x_1,\ldots,x_i,x_{i+1},\ldots,x_n]_\Lie=[x_1,\ldots,x_{i+1},x_{i},\ldots,x_n]_\Lie,$$
 thus
it is enough to count only the number of \Lie-brackets of the  form
 $[s,f_{i_1},f_{i_2},\ldots,f_{i_{n-1}}]_\Lie$ where $i_1\leq i_2\leq\cdots\leq i_{n-1}$ and $i_j\in\{1,2,\ldots,n\}$ for all
 $1\leq j\leq n-1$. In combinatorics,  the number of non-decreasing sequences  of length $s$
  from $\{1,2,\ldots,t\}$ is ${t+s-1\choose s}$  (see \cite{mar} p.125, here a sequence $\{a_n\}$ is called non-decreasing
  if $a_1\leq a_2\leq a_3\leq \cdots$). So the number of   \Lie-brackets  mentioned above will be ${2n-2 \choose n-1}$. Therefore,
 $\dim \Ima\psi\leq {2n-2\choose n-1}$ which completes the proof of (i).

(ii) Let $\dim\q=m$. As $\dim \q_{\Lie}^i=m+1-i$ for every $1\leq i\leq m+1$, we get $\dim\q_\Lie^m=1$ and $\q_\Lie^m\subseteq Z_{\Lie}(\q)$.
Moreover since $\dim\q_\Lie^{m-1}=2$, then $Z_{\Lie}(\q)=\q_\Lie^m$ and $\dim(\q_\Lie^2\cap Z_{\Lie}(\q))=1$.
Once more we get from Lemma \ref{lem1} (i) and (iv),   the following inequalities
$$\dim\mathcal{M}_\Lie(\dfrac{\q}{Z_{\Lie}(\q)})-1 \leq \dim\mathcal{M}_\Lie(\q),$$
$$\dim\mathcal{M}_\Lie(\q) \leq \dim\mathcal{M}_\Lie(\dfrac{\q}{Z_{\Lie}(\q)}),$$
respectively.
Therefore when $\q$ is \Lie-nilpotent of maximal class, the given equality
 holds for $-1\leq k\leq 0$. \ \ \ $\blacksquare$
\end{proof}

\begin{remark}\label{remark3.5} \normalfont
If  $\q$ is  an $n$-dimensional \Lie-filiform Leibniz $n$-algebra, then $\dim\q_{\Lie}^{2}=0$ i.e. $\q$ is \Lie-abelian and so
$\mathcal{M}_\Lie(\frac{\q}{Z_\Lie(\q)})=0$. Then by Theorem \ref{th0},
 $\dim\mathcal{M}_\Lie(\q)\leq \sum_{i=1}^n {n-1\choose i-1}{n\choose i}$. Also using Proposition \ref{prop1}, we have
  $\dim\mathcal{M}_\Lie(\q)\leq {2n-1\choose n}$.
\end{remark}

It is worth noting that the number of \Lie-brackets ${2n-2\choose n-1}$  ($n\geq 2$) obtained in Theorem \ref{th1} (i) gives the well-known sequence
 $2, 6, 20, 70, 252, \ldots,$
   which are called  {\it central binomial coefficients} (sequence $A000984$ in the  OEIS).
  They are actually the numbers in the central column
  (in the middle of the even-numbered rows) of
  Pascal's triangle\footnote{The pattern of numbers that forms Pascal's triangle was known well before Pascal's time (1623--1662). It had been previously investigated by many other mathematicians, including Italian algebraist $Niccol\grave{o}\ Tartaglia$ (1500--1577),
   Chinese mathematician {\it Yang Hui} (1238--1298),
    and Iranian  mathematician, astronomer and poet {\it Omar Khayyam} (1048--1131)\cite{cool}.
    This is why it is  known as  Tartaglia's triangle in Italy,
    Yang Hui's triangle in China, and  Khayyam-Pascal's triangle in Iran.}.
  Moreover, the sequence obtained by the upper bound ${2n-2\choose n-1} -1$, i.e., $1, 5, 19, 69, 251, \ldots,$
  ($A030662$ in the OEIS)
  has an interesting property.
  The numbers of this sequence are the sum of the numbers located in the rhombus part of Pascal's triangle. For instance, see three of them:
\begin{equation*}
   \begin{tabular}{>{$}l<{$\hspace{12pt}}*{13}{c}}
 &&&&&&&1&&&&&&\\
 &&&&&&1&&1\ = 5&&&&&\\
 &&&&&&&2&&&&&&
\end{tabular}
  \begin{tabular}{>{$}l<{$\hspace{12pt}}*{13}{c}}
 &&&&&&&1&&&&&&\\
 &&&&&&1&&1&&&&&\\
 &&&&&1&&2&&1\ = 19&&&&\\
 &&&&&&3&&3&&&&&\\
 &&&&&&&6&&&&&&
\end{tabular}
\end{equation*}
\begin{equation*}
\begin{tabular}{>{$}l<{$\hspace{12pt}}*{13}{c}}
 &&&&&&&1&&&&&&\\
 &&&&&&1&&1&&&&&\\
 &&&&&1&&2&&1&&&&\\
 &&&&1&&3&&3&&1\  = 69&&&\\
 &&&&&4&&6&&4&&&&\\
 &&&&&&10&&10&&&&&\\
 &&&&&&&20&&&&&&
\end{tabular}
\end{equation*}

Also, the pattern of counting the number of \Lie-brackets led us to a relationship between central binomial coefficients and triangular numbers.
{\it Triangular numbers} are created by adding up the natural numbers:
$$T_n=\sum_{i=1}^n i=\frac{n(n+1)}{2}={n+1\choose 2}.$$
 If we add up triangular numbers, we get the  {\it tetrahedral numbers} (or triangular pyramidal numbers) which can be described by an analogous formula:
$$H_n=\sum_{i=1}^n T_i=\frac{n(n+1)(n+2)}{6}={n+2\choose 3}.$$
This pattern progresses to higher dimensions as follows:
 $P_n^{(1)}=n$,\ $P_n^{(2)}=T_n$,\ $P_n^{(3)}=H_n$ and
$$P_n^{(r)}=\sum_{i=1}^n P_i^{(r-1)}=\frac{n(n+1)(n+2)\cdots(n+r-1)}{r!}={n+r-1\choose r}.$$
The next result is the  relationship mentioned above.
\begin{proposition}\label{prop0}
For every $1\leq r\leq n-1$
\begin{equation}\label{pascal}
{2n\choose n}=\sum_{i=0}^n {2n-r-1-i\choose n-r-1}P_{i+1}^{(r)}.
\end{equation}
\end{proposition}
\begin{proof} Recall that  the number of non-decreasing sequences  of length $s$
  from $\{1,2,\ldots,t\}$ is ${t+s-1\choose s}$.
 So the number of  non-decreasing sequences of length $n$
  from $\{1,2,\ldots,n+1\}$ is ${2n\choose n}$.
  On the other hand, consider the following sequence
  $$\underbrace{a_{i_1},a_{i_2},\ldots,a_{i_{n-r}}}_{(I)},\underbrace{a_{i_{n-r+1}},\ldots,a_{i_{n-1}},a_{i_n}}_{(II)},$$
where $i_1\leq i_2\leq\cdots\leq i_n$ and $1\leq i_j\leq n+1$.
   Suppose that in part (I) we put only 1, so in part (II)
    there are ${n+r\choose r}$ sequences as $s=r$ and $t=n+1$. Thus we have ${n+r\choose r}=P_{n+1}^{(r)}$ sequences in this step.

Now assume that in  (I) we put only 1, 2, so  in  (II)
there are ${n+r-1\choose r}$ sequences  as $s=r$ and $t=n$ (since here we cannot use 1, hence $t\not= n+1$).
 Also in  (I) we have ${n-r+1\choose n-r}-1$ sequences, since
$s=n-r$, $t=2$, and that $(-1)$ is because of the repetitive sequence
$\underbrace{1,1,\ldots,1}_{(n-r)-times}$
in the previous step. By Pascal's identity we have
$${n-r+1\choose n-r}-1={n-r+1\choose n-r}-{n-r\choose n-r}={n-r\choose n-r-1}.$$
Therefore, there are ${n-r\choose n-r-1}P_{n}^{(r)}$ sequences in this step.

In general,
suppose that in  (I) we put only from $\{1,2,\ldots,n-i+1\ |\ 0\leq i\leq n\}$, so in (II)
there are ${i+r\choose r}$ sequences since $s=r$ and $t=n+1-(n-i)=i+1$ (as we are not allowed to use $1,2\ldots,n-i$).
 Also in  (I) we have
 $${2n-r-i\choose n-r}-{2n-r-1-i\choose n-r}={2n-r-1-i\choose n-r-1}$$
  sequences since
$s=n-r$, $t=n-i+1$, and  we need again to remove ${2n-r-1-i\choose n-r}$ repetitive sequences in the previous step.
Thus we have
${2n-r-1-i\choose n-r-1}P_{i+1}^{(r)}$ sequences in this step, which completes the proof. \ \ \ $\blacksquare$
\end{proof}

It may be useful to know that the number of \Lie-brackets in Theorem \ref{th1} (i),  was first obtained with $r=3$, i.e., by
${2n\choose n}=\sum_{i=0}^n {2n-4-i\choose n-4}H_{i+1}$,
for every $n\geq 4$.

\begin{remark}\label{rem1} \normalfont
The interpretation of  formula (\ref{pascal})
 on Pascal's triangle can be interesting in its own way. Pay attention to the coefficients of $T_i$'s below
 and their position in Pascal's triangle:
 $${6\choose 3}={\it20}= {\bf1} T_4+{\bf1}T_3+{\bf1}T_2+{\bf1}T_1,$$
 \begin{align*}
 {8\choose 4}={\it70}&= {\bf1} T_5 +{\bf2}T_4 +{\bf3}T_3+{\bf4}T_2+{\bf5}T_1\\
 &=H_5 +H_4 + H_3 + H_2 + H_1,
 \end{align*}
 \begin{align*}
 {10\choose 5}={\it252}&= {\bf1} T_6 +{\bf3}T_5 +{\bf6}T_4 +{\bf10}T_3 +{\bf15}T_2+ {\bf21}T_1\\
 &= H_6 + 2H_5+ 3H_4 + 4H_3 + 5H_2 + 6H_1\\
 &= P_6^{(4)} + P_5^{(4)} + P_4^{(4)} + P_3^{(4)}+P_2^{(4)}+ P_1^{(4)}.\\
 \end{align*}

\begin{tabular}{>{$}l<{$\hspace{12pt}}*{16}{c}}
 &&&&&&$_{P_i^{(1)} {_\searrow}}$&&${\bf 1}_{_\swarrow}$&&&&&&\\
 &&&&&$_{T_i {_\searrow}}$&&${\bf 1}$&&${\bf1}_{_\swarrow}$&&&&&\\
 &&&&$_{H_i {_\searrow}}$&&${\bf 1}$&&${\bf2}$&&${\bf1}_{_\swarrow}$&&&&\\
 &&&$_{P_i^{(4)} {_\searrow}}$&&${\bf 1}$&&${\bf3}$&&${\bf3}$&&1&&&\\
 &&&&1&&${\bf4}$&&${\bf6}$&&4&&1&&\\
 &&&1&&${\bf5}$&&${\bf10}$&&10&&5&&1&\\
 &&1&&6&&${\bf15}$&&{\it20}&&15&&6&&1\\
 &1&&7&&${\bf21}$&&35&&35&&21&&7&&1\\
 1&&8&&28&&56&&{\it70}&&56&&28&&8&&1 \\
 &&&&&&&&\vdots
\end{tabular}

Also observe that
 \begin{align*}
 &T_1=1,\ \ T_2=3,\ \ T_3=6,\ \ T_4=10,\ \ T_5=15,\ \ T_6=21,\\
 &H_1=1,\ \ H_2=4,\ \ H_3=10,\ \ H_4=20,\ \ H_5=35,\ \ H_6=56,\\
 &P_1^{(4)}=1,\ \ P_2^{(4)}=5,\ \ P_3^{(4)}=15,\ \ P_4^{(4)}=35, \ \ P_5^{(4)}=70, \ \ P_6^{(4)}=126.
 \end{align*}
\end{remark}

In what follows, we provide some upper bounds on  $\dim\mathcal{M}_\Lie(\q)$ which are less than or equal to the existing upper bound
 published in \cite{s-b1}.
Using Lemma \ref{lem1} (iii) and  an argument similar to the proof of Theorem \ref{th1} (i), we may prove Theorem \ref{th0}
 by a different combinatorial technique.
\begin{proposition}\label{prop1}
Let $\q$ be an $m$-dimensional Leibniz $n$-algebra. Then
\[\dim\mathcal{M}_\Lie(\q)\leq {m+n-1\choose n}-\dim\q_\Lie^2.\]
\end{proposition}
\begin{proof}
Let $\q\cong\f/\R=\langle f_1+\R,f_2+\R,\ldots,f_m+\R\rangle$.
By Lemma \ref{lem1} (iii), it is enough to count the \Lie-brackets of the form
$[f_{i_1},f_{i_2},\ldots,f_{i_{n}}]_\Lie$ where $i_1\leq i_2\leq\cdots\leq i_{n}$ and $i_j\in\{1,2,\ldots,m\}$ for all
 $1\leq j\leq n$. Thus
 $$\dim\mathcal{M}_\Lie(\q) + \dim\q_\Lie^2 = \dim\dfrac{\f_\Lie^2}{[\R,\f,\ldots,\f]_\Lie}\leq {m+n-1\choose n}.$$
In fact, this upper bound is the same as that obtained in Theorem \ref{th0}, since
 $\sum_{i=1}^{n} {{n-1}\choose{i-1}}{m\choose i}=\sum_{i=1}^{n} {{n-1}\choose{n-i}}{m\choose i}={m+n-1\choose n}$. \ \ \ $\blacksquare$
\end{proof}

Also, \cite[Theorem 4.5]{s-b1} can be presented as follows.
\begin{corollary}
Let $\q$ be a Leibniz $n$-algebra such that $\dim(\q/Z_\Lie(\q))=m$. Then
\[\dim\q_\Lie^2\leq {m+n-1\choose n}.\]
\end{corollary}


Next following \cite{safa},
we  discuss an inequality on the dimension of the Schur \Lie-multiplier of \Lie-nilpotent Leibniz $n$-algebras.

\begin{theorem}\label{th2}
Let $\q$ be a finite dimensional \Lie-nilpotent  Leibniz $n$-algebra of class $c$ with  $\dim(\q/\q_\Lie^2)=m$.
 Then for every $i\geq2$ and $2\leq j\leq c$
\[\dim\mathcal{M}_\Lie(\q)\leq\dim\mathcal{M}_\Lie(\frac{\q}{\q_{\Lie}^i})+\dim\q_{\Lie}^j \bigg{[}{m+n-2\choose n-1}-1\bigg{]}.\]
\end{theorem}
\begin{proof}
 For $c=1$ or $i>c$, clearly $\q_{\Lie}^i=0$ and the above inequality holds. So let $2\leq i\leq c$, and since
 $0\not=\q_{\Lie}^c\subseteq Z_\Lie(\q)$ we may choose a one-dimensional ideal $\I$ of $\q$ contained in $\q_{\Lie}^c$.
 By Lemma \ref{lem1} (ii) and the technique applied in the proof of Theorem \ref{th1} (i) we get
 \begin{eqnarray}\label{ineq1}
 \dim\mathcal{M}_\Lie(\q)  \leq \dim\mathcal{M}_\Lie(\dfrac{\q}{\I}) +{m+n-2\choose n-1}-1.
 \end{eqnarray}
 By induction on $\dim\q$, suppose that the main inequality holds for  $\q/\I$ (whose dimension is less than $\dim\q$),  i.e.
\begin{eqnarray*}
\dim\mathcal{M}_\Lie(\dfrac{\q}{\I})&\leq&\dim\mathcal{M}_\Lie(\frac{\q/\I}{(\q/\I)_{\Lie}^i})+\dim(\q/\I)_{\Lie}^j \bigg{[}{m+n-2\choose n-1}-1\bigg{]}\\
&=&\dim\mathcal{M}_\Lie(\frac{\q}{\q_{\Lie}^i})+(\dim\q_{\Lie}^j -1)\bigg{[}{m+n-2\choose n-1}-1\bigg{]},
\end{eqnarray*}
since $\I\subseteq \q_{\Lie}^c\subseteq \q_{\Lie}^i\subseteq \q_{\Lie}^2$, \ $\dim\dfrac{\q/\I}{(\q/\I)_\Lie^2}=\dim(\q/\q_\Lie^2)=m$
 and $\dim(\q/\I)_{\Lie}^j=\dim(\q_{\Lie}^j/\I)=\dim\q_{\Lie}^j -1$.
Now using the inequality (\ref{ineq1}) we have
\begin{eqnarray*}
\dim\mathcal{M}_\Lie(\q)&\leq&\dim\mathcal{M}_\Lie(\frac{\q}{\q_{\Lie}^i})+(\dim\q_{\Lie}^j -1)\bigg{[}{m+n-2\choose n-1}-1\bigg{]}\\
&+&\bigg{[}{m+n-2\choose n-1}-1\bigg{]}\\
&=& \dim\mathcal{M}_\Lie(\frac{\q}{\q_{\Lie}^i})+\dim\q_{\Lie}^j \bigg{[}{m+n-2\choose n-1}-1\bigg{]},
\end{eqnarray*}
 which completes the proof. \ \ \ $\blacksquare$
\end{proof}


The following are  immediate corollaries of Proposition \ref{prop1} and Theorem \ref{th2}.

\begin{corollary}\label{coro11}
   Let $\q$ be a finite dimensional \Lie-nilpotent  Leibniz $n$-algebra
such that  $\dim\q_\Lie^2=\dim\q-1$. Then $\dim\mathcal{M}_\Lie(\q)\leq1$.
\end{corollary}
\begin{proof}
Theorem \ref{th2} implies that
$\dim\mathcal{M}_\Lie(\q)\leq\dim\mathcal{M}_\Lie(\q/\q_{\Lie}^2)$, and
 Proposition \ref{prop1} gives
$\dim\mathcal{M}_\Lie(\q/\q_{\Lie}^2)\leq {1+n-1\choose n}=1$. \ \ \ $\blacksquare$
\end{proof}


\begin{corollary}
  Let $\q$ be an $m$-dimensional \Lie-nilpotent  Leibniz $n$-algebra with $\dim\q_\Lie^2=d$. Then
    $$\dim\mathcal{M}_\Lie(\q)\leq {m-d+n-1\choose n}+ d {m-d+n-2\choose n-1}-d.$$
\end{corollary}

Applying Pascal's identity $d$-times, it can be shown that the above upper bound is  less than or equal to that obtained in Proposition \ref{prop1}
(and  Theorem \ref{th0}).
For $n=2$, we have the following upper bound  which   clearly improves
$\dim\mathcal{M}_\Lie(\q)\leq \frac{1}{2}m(m+1)-d,$ given in \cite{s-b1}.

\begin{corollary}\label{corollary}
  Let $\q$ be an $m$-dimensional \Lie-nilpotent  Leibniz algebra with $\dim\q_\Lie^2=d$. Then
    $$\dim\mathcal{M}_\Lie(\q)\leq \frac{1}{2}\big{(} m^2+m -md \big{)}-d.$$
\end{corollary}


The next result  provides another technique to find more upper bounds for  $\dim\mathcal{M}_\Lie(\q)$.

\begin{theorem}\label{th3}
Let $\q$ be a finite dimensional \Lie-nilpotent  Leibniz $n$-algebra of class $c\geq2$ with  $\dim(\q/\q_\Lie^2)=m$.  Then for every $2\leq i\leq c$
\[\dim\mathcal{M}_\Lie(\q)+ \dim\mathcal{M}_\Lie(\frac{\q}{\q_{\Lie}^i})\leq (\dim\q -1){m+n-2\choose n-1}.\]
\end{theorem}
\begin{proof}
Clearly $\q_\Lie^2\not=0$, since $c>1$. Now if $\dim\q=1$, then $0\not=\q_\Lie^2\subseteq\q$
implies that $\q_\Lie^2=\q$ and so $\q_\Lie^{c+1}=\q$ which is a contradiction. Thus $\dim\q>1$.
We choose a one-dimensional ideal $\I$ of $\q$ contained in $\q_\Lie^i\cap Z_\Lie(\q)$. By inequality (\ref{ineq1}) above, we have
\begin{eqnarray}\label{ineq2}
\dim\mathcal{M}_\Lie(\q)  \leq \dim\mathcal{M}_\Lie(\dfrac{\q}{\I}) +{m+n-2\choose n-1},
\end{eqnarray}
and using induction on $\dim\q$, we get
\begin{eqnarray*}
\dim\mathcal{M}_\Lie(\dfrac{\q}{\I})+\dim\mathcal{M}_\Lie(\frac{\q/\I}{(\q/\I)_{\Lie}^i})\leq(\dim\frac{\q}{\I}-1)
 {m+n-2\choose n-1},
\end{eqnarray*}
and hence
\begin{eqnarray*}
\dim\mathcal{M}_\Lie(\dfrac{\q}{\I})\leq(\dim\q-2)
{m+n-2\choose n-1}-\dim\mathcal{M}_\Lie(\frac{\q}{\q_{\Lie}^i}).
\end{eqnarray*}
Now (\ref{ineq2}) and the above inequality imply that
\begin{eqnarray*}
\dim\mathcal{M}_\Lie(\q)  &\leq& (\dim\q-2) {m+n-2\choose n-1}-\dim\mathcal{M}_\Lie(\frac{\q}{\q_{\Lie}^i}) +{m+n-2\choose n-1}\\
&=&(\dim\q-1){m+n-2\choose n-1} -\dim\mathcal{M}_\Lie(\frac{\q}{\q_{\Lie}^i}),
\end{eqnarray*}
which completes the proof. \ \ \ $\blacksquare$
\end{proof}


The following is an immediate corollary of Theorems \ref{th2} and \ref{th3}. Note that  Theorem \ref{th3}
does not work for \Lie-abelian Leibniz $n$-algebras, so in the next result $\dim\q_\Lie^2\geq 1$.

\begin{corollary}\label{corol3.15}
   Let $\q$ be an $m$-dimensional \Lie-nilpotent  Leibniz $n$-algebra with $\dim\q_\Lie^2=d\geq 1$. Then
    $$\dim\mathcal{M}_\Lie(\q)\leq \frac{1}{2}\bigg{[} \big{(} m +d -1 \big{)} {m-d+n-2\choose n-1}-d\bigg{]}.$$
\end{corollary}

If $n=2$,  we get the following upper bound which  is  smaller than $\dim\mathcal{M}_\Lie(\q)\leq \frac{1}{2}m(m+1)-d$,
but greater than the one obtained in Corollary \ref{corollary}, if $d> 2$. For $d=2$  they are the same,
and if $d=1$ then it is smaller than the one in Corollary \ref{corollary}. The following is in fact  the best upper bound for dimension of the Schur \Lie-multiplier of a \Lie-nilpotent  Leibniz $2$-algebra with one-dimensional \Lie-commutator so far.
However, the upper bound of Corollary \ref{corol3.15} is not
always smaller than the one in Proposition \ref{prop1}, for instance  when $n=3$, $d=1$ and $m\geq 10$.
\begin{corollary}\label{coro3.16}
   Let $\q$ be an $m$-dimensional \Lie-nilpotent  Leibniz algebra with $\dim\q_\Lie^2=d\geq 1$. Then
    $$\dim\mathcal{M}_\Lie(\q)\leq \frac{1}{2}(m^2-m -d^2).$$
\end{corollary}

Note that if $d=1$, then $\dim\mathcal{M}_\Lie(\q)\leq \frac{1}{2}m(m-1)-\frac{1}{2}$, and since $m(m-1)$ is always even, we get
$\dim\mathcal{M}_\Lie(\q)\leq \frac{1}{2}m(m-1)-1$.


\begin{example}\normalfont
Consider the 2-dimensional Leibniz algebra $\q=\langle x,y\rangle$ with the only non-zero bracket $[x,x]=y$.
It is easy to see that $\q^2_\Lie=Z_\Lie(\q)=\langle y\rangle$ and $\q$ is \Lie-nilpotent of class 2. By the above corollary
$\mathcal{M}_\Lie(\q)=0$. Observe that using the upper bound of \cite{s-b1} (Proposition \ref{prop1}), we get
$\dim\mathcal{M}_\Lie(\q)\leq 2$, and by Corollaries \ref{coro11}  or  \ref{corollary}, we get $\dim\mathcal{M}_\Lie(\q)\leq 1$.
\end{example}


In the next result, we discuss upper bounds for $\dim\mathcal{M}_\Lie(\q)$ in which $\q$ is \Lie-filiform and \Lie-nilpotent of maximal class,
respectively.

\begin{corollary}
$(i)$ If $\q$ is  an $m$-dimensional \Lie-filiform Leibniz $n$-algebra $(m > n)$,
then
\begin{eqnarray*}
\dim\mathcal{M}_\Lie(\q)\leq\left\{\begin{array}{ll}
\frac{m}{2}{2n-2\choose n-1}-1 &\ if \ m\leq 5,\\
\\
{2n-1\choose n}+{2n-2\choose n-1}-1 &\ if \ m\geq 6.
\end{array} \right.
\end{eqnarray*}
$(ii)$ If $\q$ is a finite dimensional \Lie-nilpotent  Leibniz $n$-algebra of maximal class, then $\dim\mathcal{M}_\Lie(\q)\leq1$.
\end{corollary}
\begin{proof}
(i) Considering  Definition \ref{def1}, and   putting  $i=2$ and $j=m-n+1$  in Theorem \ref{th2},   we get
\begin{equation}\label{eq3.15}
\dim\mathcal{M}_\Lie(\q)\leq\dim\mathcal{M}_\Lie(\frac{\q}{\q_{\Lie}^2})+ {2n-2\choose n-1}-1.
\end{equation}
Now using this inequality and Theorem \ref{th3},  we have
    $\dim\mathcal{M}_\Lie(\q)\leq \frac{1}{2}\big{[}m{2n-2\choose n-1}-1\big{]},$
     and again since ${2n-2\choose n-1}$ is even for all $n\geq2$,
so $\dim\mathcal{M}_\Lie(\q)\leq \frac{m}{2}{2n-2\choose n-1}-1$.

On the other hand, inequality (\ref{eq3.15})  and Proposition \ref{prop1}  imply that
$\dim\mathcal{M}_\Lie(\q)\leq {2n-1\choose n}+ {2n-2\choose n-1}-1.$
Therefore,
\[\dim\mathcal{M}_\Lie(\q) \leq\min\bigg{\{} \frac{m}{2}{2n-2\choose n-1}-1 , \ {2n-1\choose n}+{2n-2\choose n-1}-1 \bigg{\}}.\]
Using the Pascal's identity ${2n-1\choose n}={2n-2\choose n}+{2n-2\choose n-1}$,
 one can easily check that if $m\leq 5$,  the first upper bound is less than or equal to the other one
 (equality occurs only when $m=5$ and $n=2$), and if $m\geq 6$, the second one
  is smaller.

(ii) It follows from Definition \ref{def2} and Corollary \ref{coro11}. \ \ \ $\blacksquare$
\end{proof}

In the case  $n=2$, we have:

\begin{corollary}
 If $\q$ is  an $m$-dimensional \Lie-filiform Leibniz algebra $(m>2)$,
then
\begin{eqnarray*}
\dim\mathcal{M}_\Lie(\q)\leq\left\{\begin{array}{ll}
m-1 &\ if \ m\leq 5,\\
4 &\ if \ m\geq 6.
\end{array} \right.
\end{eqnarray*}
Also if $m=2$, then $\dim\mathcal{M}_\Lie(\q)=3$.
\end{corollary}
\begin{proof}
If $m=2$, then $\q$ is \Lie-abelian (Remark \ref{remark3.5}), and by Theorem \ref{theorem2.4} we get $\dim\mathcal{M}_\Lie(\q)=3$. \ \ \ $\blacksquare$
\end{proof}

The following example shows that the above  corollary provides a better  upper bound than the previous results for \Lie-filiform Leibniz algebras.
Note that  for a \Lie-filiform Leibniz algebra of dimension $3$ (in this case $\dim\q^2_\Lie=1$), the above corollary and Corollary \ref{coro3.16}
give the same upper bound $\dim\mathcal{M}_\Lie(\q)\leq 2$. According to Definition \ref{def1}, for \Lie-filiform Leibniz algebras if $m\geq 4$, then $\dim\q^2_\Lie\geq 2$.
 So, the bound $\dim\mathcal{M}_\Lie(\q)\leq \frac{1}{2}m(m-1)-1$
 in Corollary \ref{coro3.16} is still the best upper bound ever presented, if $\q$ is a \Lie-nilpotent Leibniz algebra with $\dim\q^2_\Lie=1$.

\begin{example}\normalfont
Consider the 4-dimensional Leibniz algebra $\q=\langle x_1,x_2,x_3,x_4\rangle$ with  non-zero brackets
 $[x_1,x_1]=x_3$, $[x_1,x_2]=x_4$, $[x_2,x_1]=x_3$ and $[x_3,x_1]=x_4$ (see \cite[Theorem 3.2]{alb}).
 One can easily check that $\q^2_\Lie=\langle x_3,x_4\rangle$, $\q^3_\Lie=\langle x_4\rangle$ and $\q^4_\Lie=0$, i.e.
 $\q$ is  \Lie-filiform. So, by the above corollary $\dim\mathcal{M}_\Lie(\q)\leq 3$, while using
  Corollaries \ref{corollary} or \ref{coro3.16},
  we get $\dim\mathcal{M}_\Lie(\q)\leq 4$, and by Proposition \ref{prop1}, $\dim\mathcal{M}_\Lie(\q)\leq 8$.
\end{example}


In Theorems \ref{th2} and \ref{th3}, induction is used on the dimension of $\q$ while in the next, following \cite{yan},
 it is used on the \Lie-nilpotency class of $\q$. Moreover, the following result  improves Theorem \ref{th2}.

\begin{theorem}
Let $\q$ be a finite dimensional \Lie-nilpotent  Leibniz $n$-algebra of class $c$ with  $\dim\dfrac{\q/Z_\Lie(\q)}{(\q/Z_\Lie(\q))_\Lie^2}=m$.
  Then for every  $i\geq 2$ and $2\leq j\leq c$
\[\dim\mathcal{M}_\Lie(\q)\leq\dim\mathcal{M}_\Lie(\frac{\q}{\q_{\Lie}^i})+\dim\q_{\Lie}^{j} \bigg{[}{m+n-2\choose n-1}-1\bigg{]}.\]
\end{theorem}
\begin{proof}
If $c=1$ or $i>c$, then the result holds.
 Assume that $2\leq i\leq c$  and
the result holds for \Lie-nilpotent Leibniz $n$-algebras of class less than $c$.
Clearly, $0\not=\q_{\Lie}^c\subseteq Z_\Lie(\q)$, \  $(\q/\q_{\Lie}^c)_\Lie^i=\q_{\Lie}^i/\q_{\Lie}^c$ and
$Z_\Lie(\q)/\q_{\Lie}^c\subseteq Z_\Lie(\q/\q_\Lie^c)$. Put $A=\dfrac{\q/\q_\Lie^c}{Z_\Lie(\q/\q_\Lie^c)}$,
\ $B=\q/Z_\Lie(\q)=\dfrac{\q/\q_\Lie^c}{Z_\Lie(\q)/\q_\Lie^c}$ and $\dim(A/A_\Lie^2)=m'$. It is easy to check that $m'\leq \dim(B/B_\Lie^2)=m$.
By induction hypothesis
\begin{eqnarray}\label{ineq3}
\dim\mathcal{M}_\Lie(\frac{\q}{\q_{\Lie}^c})&\leq&\dim\mathcal{M}_\Lie(\frac{\q/\q_{\Lie}^c}{\q_{\Lie}^i/\q_{\Lie}^c})  +
\dim\frac{\q_{\Lie}^j}{\q_{\Lie}^c}
\bigg{[}{m'+n-2\choose n-1}-1\bigg{]}\nonumber\\
&\leq&\dim\mathcal{M}_\Lie(\frac{\q}{\q_{\Lie}^i})  +
\dim\frac{\q_{\Lie}^j}{\q_{\Lie}^c}
\bigg{[}{m+n-2\choose n-1}-1\bigg{]}.
\end{eqnarray}
Now, similar to \cite[Theorem 3.6]{esh} and the technique used earlier, one can show that
\begin{eqnarray}\label{ineq4}
\dim\mathcal{M}_\Lie(\q)\leq\dim\mathcal{M}_\Lie(\frac{\q}{\q_{\Lie}^c})+\dim\q_{\Lie}^c \bigg{[}{m''+n-2\choose n-1}-1\bigg{]},
\end{eqnarray}
where $m''=\dim(\q/Z^\Lie_{c-1}(\q))$ and $Z^\Lie_{i}(\q)$ is the $i$th term
of the upper \Lie-central series of $\q$  (see Definition 10 and Theorem 4 of \cite{c-k}).
Since $\q_\Lie^2\subseteq Z^\Lie_{c-1}(\q)$, hence
\begin{eqnarray}\label{ineq5}
m''=\dim(\q/Z^\Lie_{c-1}(\q))\leq \dim(\q/(\q_\Lie^2 + Z_\Lie(\q)))=\dim(B/B_\Lie^2)=m.
\end{eqnarray}
Thus  inequalities (\ref{ineq3}), (\ref{ineq4}) and (\ref{ineq5}) imply that
\begin{eqnarray*}
\dim\mathcal{M}_\Lie(\q)&\leq&\dim\mathcal{M}_\Lie(\frac{\q}{\q_{\Lie}^i})  +
\dim\frac{\q_{\Lie}^j}{\q_{\Lie}^c}
\bigg{[}{m+n-2\choose n-1}-1\bigg{]}\\
&+&\dim\q_{\Lie}^c \bigg{[}{m''+n-2\choose n-1}-1\bigg{]}\\
&\leq& \dim\mathcal{M}_\Lie(\frac{\q}{\q_{\Lie}^i})  +
\dim\frac{\q_{\Lie}^j}{\q_{\Lie}^c}
\bigg{[}{m+n-2\choose n-1}-1\bigg{]}\\
&+&\dim\q_{\Lie}^c  \bigg{[}{m+n-2\choose n-1}-1\bigg{]}\\
&\leq&\dim\mathcal{M}_\Lie(\frac{\q}{\q_{\Lie}^i})  +
\dim\q_{\Lie}^j
\bigg{[}{m+n-2\choose n-1}-1\bigg{]},
\end{eqnarray*}
which completes the proof. \ \ \ $\blacksquare$
\end{proof}


 {\bf Narcisse G. Bell Bogmis}\\
 Department of Mathematics, Faculty of Sciences, University of Dschang,
 Dschang, Cameroon.\\
E-mail address:  bellnarcisse3@gmail.com \\

{\bf Guy R. Biyogmam} \\
Department of Mathematics, Georgia College \& State University, Milledgeville, GA, USA.\\
E-mail address:  guy.biyogmam@gcsu.edu \\

{\bf Hesam Safa} \\
Department of Mathematics, Faculty of Basic Sciences, University of Bojnord, Bojnord, Iran. \\
E-mail address:  hesam.safa@gmail.com, \ h.safa@ub.ac.ir  \\

 {\bf  Calvin Tcheka}\\
 Department of Mathematics, Faculty of Sciences, University of Dschang,
 Dschang, Cameroon.\\
E-mail address:  calvin.tcheka@univ-dschang.org \\

\end{document}